\documentclass[12pt,letterpaper]{article}
\usepackage{amsmath,amsfonts,amsthm,amssymb}

\swapnumbers %Theorems, lemmas, etc. numbered at the left.

\newtheorem{lemma}{Lemma}[section]
\newtheorem{corollary}[lemma]{Corollary}
\newtheorem{theorem}[lemma]{Theorem}

\theoremstyle{definition} %Set Following proclamations in normal text.

\newcommand\reals{{\mathbb R}}

\newcommand{\dg}{\sp{\text{\rm o}}}

\newcommand{\spec}{\operatorname{spec}}

\begin{document}

\title{Kadison's antilattice theorem for a synaptic algebra}

\author{David J. Foulis{\footnote{Emeritus Professor, Department of
Mathematics and Statistics, University of Massachusetts, Amherst,
MA; Postal Address: 1 Sutton Court, Amherst, MA 01002, USA;
foulis@math.umass.edu.}}\hspace{.05 in},  Sylvia
Pulmannov\'{a}{\footnote{ Mathematical Institute, Slovak Academy of
Sciences, \v Stef\'anikova 49, SK-814 73 Bratislava, Slovakia;
pulmann@mat.savba.sk. The second  author was supported by grant
VEGA No.2/0069/16.}}}

\date{}

\maketitle

\begin{abstract}
We prove that if $A$ is a synaptic algebra and the orthomodular
lattice $P$ of projections in $A$ is complete, then $A$ is a
factor iff $A$ is an antilattice. We also generalize several
other results of R. Kadison pertaining to infima and suprema
in operator algebras.
\end{abstract}

\section{Introduction} \label{sc:Intro}

A synaptic algebra \cite{FSyn, FPproj, FPtype, FPsymSA, FPcom, FPBanach,
FJP2proj, FJPpande, vectlat, FJPstat, FJPMSR, FJPLS, Pid} is a generalization of
the self-adjoint part of several structures based on operator algebras.
For instance, although a synaptic algebra $A$ need not be norm complete
(i.e., Banach), $A$ is isomorphic to the self-adjoint part of a Rickart
C$\sp{\ast}$-algebra if and only if it is Banach \cite[Theorem 5.3]
{FPBanach}. Also, $A$ is isomorphic to the self-adjoint part of an
AW$\sp{\ast}$-algebra iff it is Banach and its projection lattice is
complete \cite[Theorem 8.5]{FPBanach}. Numerous additional examples of
synaptic algebras can be found in the references cited above.

In \cite{Kad}, Richard Kadison calls the self-adjoint part $\mathfrak S$
of an operator algebra an \emph{antilattice} if and only if, whenever two
elements of $\mathfrak S$ have an infimum in $\mathfrak S$, then the
elements are comparable (i.e., one is less than or equal to the other).
As Kadison remarks \cite[ p. 505]{Kad}, ``A moments thought shows that
this is as strongly nonlattice as a partially ordered vector space can
be." He shows that, in important cases, the condition that $\mathfrak S$
is an antilattice is equivalent to the condition that the operator algebra
in question is a factor (i.e., its center consists only of scalars). Our
main theorem in this paper (Theorem \ref{th:MainTh}) is a version of Kadison's
result for a synaptic algebra in which the projections form a complete lattice.

\section{Some properties of a synaptic algebra}

Axioms for a synaptic algebra can be found in \cite[\S 1]{FSyn} and
will not be repeated here. Rather, we shall briefly sketch some of the
important features of a synaptic algebra that we shall need below.
Readers who are informed about operator algebras will be familiar with
many of these features---details can be found in the references given
in Section \ref{sc:Intro}, especially in \cite{FSyn}. We use `iff' as
an abbreviation for `if and only if,' the notation $:=$ means `equals
by definition,' and $\reals$ is the ordered field of real numbers.

\emph{In what follows, we assume that $A$ is a synaptic algebra} \cite
[Definition 1.1]{FSyn}. Thus, associated with $A$ is a real or complex
associative unital algebra $R$ with unit $1$ called the \emph{enveloping
algebra} of $A$ such that (1) $1\in A$, (2) $A$ is a real linear subspace
of $R$, (3) the linear space $A$ is a partially ordered order-unit normed
space with order unit $1$ and positive cone $A\sp{+}:=\{a\in A:0\leq a\}$,
and (4) the order-unit norm of $a\in A$ is denoted and defined by $\|a\|
:=\inf\{0<\lambda\in\reals:-\lambda1\leq a\leq\lambda1\}$ \cite[pp. 67--69]
{Alf}. In important examples, the enveloping algebra $R$ is a unital
(concrete or abstract) operator algebra with an adjoint mapping $a\mapsto
a\sp{\ast}$, $A$ is the self-adjoint part of $R$, and $A\sp{+}
=\{rr\sp{\ast}:r\in R\}$. We shall assume that $A$ is \emph{nontrivial},
i.e., $A\not=\{0\}$. Then $1\not=0$, which enables us, as usual, to
identify each scalar $\lambda\in\reals$ with the element $\lambda1\in A$.

Let $M\subseteq A$. Then $M$ is understood to be partially ordered
under the restriction of the partial order $\leq$ on $A$. If $m,n
\in M$, then the infimum (greatest lower bound) and the supremum
(least upper bound) of $m$ and $n$ in $M$---if they exist---are
written as $m\wedge\sb{M}n$ and $m\vee\sb{M}n$, respectively. An
\emph{involution} on $M$ is a mapping $m\mapsto m\sp{\prime}$ that
is order reversing (i.e., $m\leq n\Rightarrow n\sp{\prime}\leq m
\sp{\prime}$) and of order two (i.e., $(m\sp{\prime})\sp{\prime}=m$).
Such an involution provides a ``duality" between existing infima and
suprema in $M$ according to $m\wedge\sb{M}n=(m\sp{\prime}\vee\sb{M}n
\sp{\prime})\sp{\prime}$ and $m\vee\sb{M}n=(m\sp{\prime}\wedge\sb{M}n
\sp{\prime})\sp{\prime}$. Note that the mapping $a\mapsto -a$ is an
involution on $A$ itself.

Let $a,b\in A$. Then the product $ab$ is understood to be calculated
in $R$ and may or may not belong to $A$; however, it is assumed that
$0\leq a\sp{2}\in A$. Consequently, the Jordan product
\[
a\odot b:=\frac12(ab+ba)=\frac12((a+b)\sp{2}-a\sp{2}-b\sp{2})\in A,
\]
so $A$ is a real unital special Jordan algebra under $\odot$ \cite{McC}.
If $a$ commutes with $b$, i.e., $ab=ba$ in $R$, we write $aCb$. Evidently,
if $aCb$, then $ab=a\odot b\in A$. It can be shown that $ab=0$ iff $ba=0$.
Also, if $a,b\in A\sp{+}$ and $aCb$, then $ab\in A\sp{+}$. As a consequence,
if $a,b,c\in A$, $a\leq b$, $0\leq c$, $cCa$, and $cCb$, then $ca\leq cb$;
indeed, by the hypotheses, $0\leq c, b-a$ and $cC(b-a)$, whence $0\leq
cb-ca$, i.e., $ca\leq cb$.

Suppose that $a,b\in A$ and put $c:=2(a\odot b)$. Then $aba=a\odot c-
a\sp{2}\odot b\in A$ and the mapping $b\mapsto aba$ is called the
\emph{quadratic mapping} on $A$ determined by $a$. It can be shown
that the quadratic mapping $b\mapsto aba$ is both linear and order
preserving on $A$.

An element $e\in A$ such that $0\leq e\leq 1$ is called an \emph{effect},
and it can be shown that $e\in A$ is an effect iff $e\sp{2}\leq e$.
The subset $E:=\{e\in A:0\leq e\leq 1\}$ of $A$ forms a convex effect
algebra \cite{GPBB} under the partially defined binary operation obtained
by restriction to $E$ of the addition operation $+$ on $A$. The mapping
$e\mapsto e\sp{\perp}:=1-e$,  called the \emph{orthosupplementation} on
$E$, is an involution on $E$.

An idempotent element $p=p\sp{2}$ in $A$ is called a \emph{projection},
and the set $P:=\{p\in A:p=p\sp{2}\}$ of all projections in $A$ is a
subset of $E$; in fact, $P$ is precisely the set of all extreme points
of the convex set $E$. Under the restriction to $P$ of the partial order
on $A$, $P$ forms an orthomodular lattice (OML) \cite[\S 5]{FSyn}, \cite
{Beran, Kalm}, and the orthocomplementation $p\mapsto p\sp{\perp}:=1-p$
on the OML $P$ is the restriction to $P$ of the orthosupplementation on
$E$. The orthocomplementation mapping is an involution on $P$, whence we
have an infimum-supremum duality on the lattice $P$. If $e\in E$ and $p
\in P$, it can be shown that $e\leq p\Leftrightarrow e=ep\Leftrightarrow
e=pe$ and $p\leq e\Leftrightarrow p=pe\Leftrightarrow p=ep$.

Let $p,q\in P$. Then $p\leq q$ iff $p=pq$ iff $p=qp$. We shall write the
infimum and the supremum of $p$ and $q$ in the lattice $P$ as $p\wedge q$
and $p\vee q$, respectively, (without subscripts on $\wedge$ and $\vee$).
If $pCq$, then $p\wedge q=pq=qp$ and $p\vee q=p+q-pq$. The projections $p$
and $q$ are said to be \emph{orthogonal}, in symbols $p\perp q$, iff
$p\leq q\sp{\perp}$, or equivalently iff $q\leq p\sp{\perp}$. Note that
$p\perp q$ iff $pq=0$ iff $qp=0$ iff $p+q=p\vee q$. In particular, for the
orthocomplement $p\sp{\perp}$ of $p$ in $P$, we have $p\perp p\sp{\perp}$,
$pCp\sp{\perp}$, $p\wedge p\sp{\perp}=pp\sp{\perp}=0$ and $p\vee p\sp{\perp}
=p+p\sp{\perp}=1$.

If $M\subseteq A$, then $C(M):=\{a\in A:aCm\text{\ for all\ }m\in M\}$
is called the \emph{commutant} of $M$ and $CC(M):=C(C(M))$ is called
the \emph{bicommutant} of $M$. Evidently, $C(M)$ is a linear subspace
(or a vector subspace) of $A$. The subset $M$ of $A$ is said to be
\emph{commutative} iff $mCn$ for all $m,n\in M$, i.e., iff $M\subseteq
C(M)$. If $M$ is commutative, then so is $CC(M)$. If $a\in A$, then
$C(a):=C(\{a\})$ and $CC(a):=CC(\{a\})$.

If $a\in A\sp{+}$, then there exists a unique $a\sp{1/2}\in A\sp{+}$,
called the \emph{square root} of $a$, such that $(a\sp{1/2})\sp{2}=a$;
moreover, $a\sp{1/2}\in CC(a)$. The \emph{absolute value} of $a\in A$
is denoted and defined by $|a|:=(a\sp{2})\sp{1/2}$ and the \emph{positive}
and \emph{negative} parts of $a$ are denoted and defined by $a\sp{+}:=
\frac12(|a|+a)$ and $a\sp{-}:=\frac12(|a|-a)$, respectively. Then $0\leq
a\sp{+}, a\sp{-}\in CC(a)$, $a=a\sp{+}-a\sp{-}$, $|a|=a\sp{+}+a\sp{-}$,
and $a\sp{+}a\sp{-}=0$. It turns out that an element $a\in A$ has an
\emph{inverse} $a\sp{-1}$ in $A$ such that $aa\sp{-1}=a\sp{-1}a=1$ iff
there exists $0<\epsilon \in\reals$ such that $\epsilon\leq|a|$.

If $a\in A$, there exists a unique projection $a\dg\in P$, called the
\emph{carrier} of $a$, such that, for all $b\in A$, $ab=0\Leftrightarrow
a\dg b=0$. (Some authors would refer to $a\dg$ as the \emph{support} of
$a$.) It turns out that $a\dg\in CC(a)$, $a\dg$ is the smallest
projection $p\in P$ such that $a=ap$, and if $e\in E$, then $e\dg$ is
the smallest projection $p\in P$ such that $e\leq p$. See \cite[Theorem
2.10]{FSyn} for additional properties of the carrier.

\begin{lemma} \label{le:sup,inf,p,q}
If $p,q\in P$, then $p\wedge q=p\wedge\sb{E}q$ and $p\vee q=p\vee\sb{E}q$.
\end{lemma}

\begin{proof}
We have $p,q,p\wedge q\in P\subseteq E$ and $p\wedge q\leq p,q$. Suppose
$e\in E$ with $e\leq p,q$. Then $e\dg\leq p,q$, whence $e\leq e\dg\leq
p\wedge q$, and therefore $p\wedge q=p\wedge\sb{E}q$. By duality, $p\vee
q=p\vee\sb{E}q$.
\end{proof}

In view of Lemma \ref{le:sup,inf,p,q}, no confusion will result if we
use the same notation $e\wedge f$ and $e\vee f$ (without subscripts)
for existing infima and suprema of effects as we do for infima and
suprema of projections. (Actually, the question of just when two effects
have an infimum in $E$ is important, but not easy to resolve \cite{GGJinf}.)

There is a very satisfactory spectral theory for $A$ based on the following
notions \cite[\S 8]{FSyn}: Let $a\in A$. The \emph{spectral resolution} of
$a$ is the one-parameter family of projections $\{p\sb{\lambda}:\lambda
\in\reals\}$ given by $p\sb{\lambda}:=((a-\lambda)\sp{+})\dg)\sp{\perp}
\in CC(a)$ for all $\lambda\in\reals$. (Recall that $\lambda$ is identified
with $\lambda1\in A$.) The \emph{spectral lower and upper bounds} for $a$
are defined by $L:=\sup\{\lambda\in\reals:\lambda\leq a\}$ and $U:=\inf
\{\lambda\in\reals:a\leq\lambda\}$, respectively. By \cite[Theorem 3.1]
{SROUS}, we have $-\infty<L\leq a\leq U<\infty$ and $\|a\|=\max\{|L|,|U|\}$.

A real number $\rho$ belongs to the \emph{resolvent
set} of $a$ iff there is an open interval $I$ in $\reals$ such that $\rho
\in I$ and $p\sb{\lambda}=p\sb{\rho}$ for all $\lambda\in I$. The \emph
{spectrum} of $a$, in symbols $\spec(a)$, which is defined to be the
complement in $\reals$ of the resolvent set of $a$, is a nonempty closed
and bounded subset of $\reals$ with all of the expected properties.

\begin{theorem} \label{th:qsublambda}
Let $a\in A$ with $0<a$ and let $\{p\sb{\lambda}:\lambda\in\reals\}$
be the spectral resolution of $a$. For $\lambda\in\reals$, define
$q\sb{\lambda}:=1-p\sb{\lambda}=((a-\lambda)\sp{+})\dg$. Then{\rm: (i)}
$q\sb{\lambda}\in CC(a)\subseteq C(a)$. {\rm (ii)} If $\lambda<0$, then
$q\sb{\lambda}=1$. {\rm(iii)} $0<q\sb{0}=a\dg$. {\rm(iv)} If $0<\lambda
<\|a\|$, then $0<q\sb{\lambda}$ and $\lambda q\sb{\lambda}\leq q
\sb{\lambda}a=aq\sb{\lambda}\leq a$. {\rm(v)} If $\lambda\geq\|a\|$,
then $q\sb{\lambda}=0$.
\end{theorem}

\begin{proof}
Part (i) follows from \cite[Theorem 8.4 (i)]{FSyn}. Let $L$ and $U$
be the lower and upper spectral bounds for $a$. Since $0<a$, we have
$0\leq L$, whence if $\lambda<0$, then $\lambda<L$, and (ii) then
follows from \cite[Theorem 8.4 (vi)]{FSyn}. Since $0<a$, we have $a=a\sp{+}$,
whence $q\sb{0}=((a-0)\sp{+})\dg=a\dg\geq 0$. Also, $a\dg\not=0$, else
$a=0$, contradicting $0<a$, and we have (iii).

As $0\leq L\leq U$ and $\|a\|=\max\{|L|,|U|\}$, we have $U=\|a\|$. Suppose
$0<\lambda<\|A\|=U$. Then by \cite[Theorem 8.4 (ii)]{FSyn},
\setcounter{equation}{0}
\begin{equation} \label{eq:projectoid1}
(1-q\sb{\lambda})(a-\lambda)\leq 0\leq q\sb{\lambda}(a-\lambda).
\end{equation}
Thus, if $q\sb{\lambda}=0$, then $a\leq\lambda$, so $-\lambda\leq a
\leq\lambda$, whence $\|a\|\leq\lambda$, contradicting $\lambda<\|a\|$.
Therefore, $0<q\sb{\lambda}$, and by (\ref{eq:projectoid1}), we have
$\lambda q\sb{\lambda}\leq q\sb{\lambda}a=aq\sb{\lambda}$. Since
$(1-q\sb{\lambda})Ca$ and $0\leq 1-q\sb{\lambda}, a$, it follows that
$0\leq(1-q\sb{\lambda})a$, whence we also have $q\sb{\lambda}a\leq a$,
and therefore (iv) holds. Finally, (v) follows from \cite
[Theorem 8.4 (v)]{FSyn}.
\end{proof}

\begin{corollary} \label{co:subprojectionoid}
If $0<a\in A$, then there exists $0<\lambda\in\reals$ and $0<p\in P$
such that $\lambda p\leq a$.
\end{corollary}

\begin{proof}
Choose $\lambda\in\reals$ with $0<\lambda<\|a\|$, and in Theorem
\ref{th:qsublambda} let $p=q\sb{\lambda}$.
\end{proof}

An element $s\in A$ such that $s\sp{2}=1$ is called a \emph{symmetry}, and
two elements $a,b\in A$ are said to be \emph{exchanged} by the symmetry
$s$ iff $sas=b$ (or, equivalently, iff $sbs=a$) \cite{FPsymSA}. An element
$t\in A$ is called a \emph{partial symmetry} iff $t\sp{2}=p\in P$, and
$a$ and $b$ are \emph{exchanged} by $t$ iff $tat=b$ and $tbt=a$. There is
a bijective correspondence $s\leftrightarrow p$ between symmetries $s\in A$
and projections $p\in P$ given by $s=2p-1$ and $p=\frac12(s+1)$.

\begin{lemma} \label{le:SymIneq}
Let $s\in A$ be a symmetry, $s\not=-1$, and $\lambda,\mu\in\reals$. Then
$\lambda s\leq\mu\Rightarrow\lambda\leq\mu$.
\end{lemma}

\begin{proof}
Assume the hypotheses. Suppose $\lambda s\leq\mu$, and let $p=\frac12(s+1)$
be the projection corresponding to $s$. Then $s=2p-1=p-p\sp{\perp}$, and
since $s\not=-1$, we have $0<p$. Then $\lambda p-\lambda p\sp{\perp}
\leq\mu$, so $\lambda p=(\lambda p-\lambda p\sp{\perp})p\leq\mu p$, i.e.,
$0\leq(\mu-\lambda)p$. If $\mu-\lambda<0$, than $(\mu-\lambda)p<0$, contradicting
$0\leq(\mu-\lambda)p$, whence $0\leq\mu-\lambda$, so $\lambda\leq\mu$.
\end{proof}

A subset $S\subseteq A$ is called a \emph{sub-synaptic algebra} of $A$ iff
$S$ is a linear subspace of $A$, $1\in S$, and $S$ is closed under the
formation of squares, square roots, carriers, and inverses, in which case
$S$ is a synaptic algebra in its own right. For instance, if $M\subseteq A$,
then $C(M)$ is a sub-synaptic algebra of $A$. In particular, if $M$ is a
commutative subset of $A$, then $CC(M)$ is a commutative sub-synaptic algebra
of $A$.

If $p\in P$, then $pAp:=\{pap:a\in A\}=\{a\in A:a=pa=ap\}$ is a linear
subspace of $A$ that is closed under the formation of squares, square roots,
carriers, and inverses; it is a synaptic algebra in its own right with $p$ as
its unit element and with $\{prp:r\in R\}$ as its enveloping algebra. The OML
of projections in $pAp$ is the interval $P[0,p]:=\{q\in P:q\leq p\}$. If $t$
is a partial symmetry in $A$ with $t\sp{2}=p$, then $t$ is a symmetry in
$pAp$.

The \emph{center} of $A$ is the commutative sub-synaptic algebra $C(A)$.
If $C(A)=\reals$, then $A$ is called a \emph{factor}. For instance, the
self-adjoint part ${\mathcal B}\sp{sa}({\mathfrak H})$ of the unital
C$\sp{\ast}$-algebra of all bounded linear operators on a Hilbert space
${\mathfrak H}$ is a factor.

\begin{theorem} \label{th:factorcondition}
The synaptic algebra $A$ is a factor iff the only projections in $C(A)$ are
$0$ and $1$.
\end{theorem}

\begin{proof}
If $C(A)=\reals$, then clearly the only projections in $C(A)$ are $0$
and $1$. Conversely, suppose that the only projections in $C(A)$ are $0$
and $1$ and let $a\in C(A)$. By \cite[Theorem 8.10]{FSyn}, the spectral
resolution $\{p\sb{\lambda}:\lambda\in\reals\}$ of $a$ is contained in
$C(A)\cap P$, whence $p\sb{\lambda}\in\{0,1\}$ for all $\lambda\in\reals$.
Therefore, by \cite[Theorem 8.4 (iii) and (vii)]{FSyn}, there exists
$\alpha\in\reals$ such that $p\sb{\lambda}=0$ for $\lambda<\alpha$ and
$p\sb{\lambda}=1$ for $\alpha\leq\lambda$, and it follows that $\spec(a)=
\{\alpha\}$. Consequently, by \cite[Theorem 8.9]{FSyn}, $a=\alpha1=\alpha$.
\end{proof}

\section{Two commuting projections} \label{sc:CommutingProjs}

\begin{lemma} {\rm\cite[Lemma 2]{Kad}} \label{lm:KadL2}
If $p,q\in P$, then $p\wedge q$ is also the infimum of $p$ and $q$
in $A\sp{+}$.
\end{lemma}

\begin{proof}
Clearly, $p\wedge q\in A\sp{+}$. Suppose that $a\in A\sp{+}$ and $a\leq p,q$.
Then $0\leq a\leq p\leq 1$, so $a\in E$, and since $p\wedge q$ is also the
infimum of $p$ and $q$ in $E$ (Lemma \ref{le:sup,inf,p,q}), we have $a\leq p
\wedge q$. Therefore, $p\wedge q$ is the infimum of $p$ and $q$ in $A\sp{+}$.
\end{proof}

\begin{lemma} {\rm(Cf. \cite[Proof of Theorem 1]{Kad})} \label{Le:SeeTheorem1Kad}
Suppose that $V$ is a vector subspace of $A$;\ $p,q\in P\cap V$;\ $p\wedge q
\in V$;\ there exists $w\in P\cap V$ with $p,q\leq w$;\ and $p\wedge\sb{V}q$
exists. Put $p\sb{1}:=p-p\wedge\sb{V}q$ and $q\sb{1}:=q-p\wedge\sb{V}q$.
Then{\rm:}
\begin{enumerate}
\item $p\wedge\sb{V}q=p\wedge q\in P\cap V$.
\item $p\sb{1}, q\sb{1}\in P\cap V$ with $p\sb{1}\wedge\sb{V}q\sb{1}=0$.
\item $p\sb{1}q\sb{1}=q\sb{1}p\sb{1}=0$.
\item $p$ commutes with $q$.
\end{enumerate}
\end{lemma}

\begin{proof}
Put $g:=p\wedge q\in P\cap V$ and $g\sb{V}:=p\wedge\sb{V}q\in V$. By Lemma
\ref{lm:KadL2}, $g$ is the infimum of $p$ and $q$ in $A\sp{+}$; hence,
\setcounter{equation}{0}
\begin{equation} \label{eq:Kad1}
\text{if\ }a\in A\sp{+}\text{\ and\ }a\leq p, q\text{\ then\ } a\leq g.
\end{equation}
Since $0\in V$ and $0\leq p, q$, it follows that $0\leq g\sb{V}$, hence by
(\ref{eq:Kad1}) with $a:=g\sb{V}$, we have $g\sb{V}\leq g$. Also, $g\in V$
and $g\leq p, q$, so $g\leq g\sb{V}$. Thus, $g\sb{V}=g$, and we have (i).

By (i), $g\sb{V}\in P\cap V$, and since $g\sb{V}\leq p$, it follows
that $p\sb{1}=p-g\sb{V}\in P\cap V$. Likewise, $q\sb{1}=q-g\sb{V}\in P
\cap V$. Evidently $0\leq p\sb{1}, q\sb{1}\in V$. Suppose that $v\in V$ and
$v\leq p\sb{1},q\sb{1}$. Then $v+g\sb{V}\in V$ with $v+g\sb{V}\leq p,q$,
whence $v+g\sb{V}\leq g\sb{V}$, so $v\leq 0$, and we have $p\sb{1}\wedge\sb{V}
q\sb{1}=0$. This completes the proof of (ii).

Since $p\sb{1}\leq p\leq w$ and $q\sb{1}\leq q\leq w$, we have $p\sb{1}-w,
q\sb{1}-w\leq 0$, whence $p\sb{1}+q\sb{1}-w\leq p\sb{1},q\sb{1}$. Therefore,
$p\sb{1}+q\sb{1}-w\leq p\sb{1}\wedge\sb{V}q\sb{1}=0$, whence $p\sb{1}\leq
w-q\sb{1}$. As $w$ and $q\sb{1}$ are projections and $q\sb{1}\leq w$, it
follows that $w-q\sb{1}$ is a projection, whereupon $p\sb{1}=p\sb{1}
(w-q\sb{1})=p\sb{1}w-p\sb{1}q\sb{1}=p\sb{1}-p\sb{1}q\sb{1}$, so $p\sb{1}
q\sb{1}=0$, and we have (iii).

By (i), we have $p\sb{1}=p-p\wedge q$, so $p\wedge q=p-p\sb{1}\leq 1-p\sb{1}$,
and it follows that $(p\wedge q)Cp\sb{1}$. Likewise, $(p\wedge q)Cq\sb{1}$.
By (iii), $p\sb{1}Cq\sb{1}$, and therefore $(p\sb{1}+p\wedge q)C(q\sb{1}+
p\wedge q)$, i.e., $pCq$, and we have (iv).
\end{proof}

\begin{theorem} {\rm(Cf. \cite[Theorem 1]{Kad})} \label{th:KadTh1}
Two projections $p$ and $q$ in $A$ commute iff there exists a vector subspace
$V$ of $A$ such that $p,\,q,\,p\wedge q\in V$, there is a projection $w\in V
\cap P$ with $p,q\leq w$, and the infimum $p\wedge\sb{V}q$ in $V$ of the two
projections exists.
\end{theorem}

\begin{proof}
In view of Lemma \ref{Le:SeeTheorem1Kad}, we have only to prove that if $p$ and
$q$ commute, then there is a vector subspace $V$ of $A$ with the indicated
properties. It suffices to take $V:=CC(\{p,q\})$, noting that $V$ is commutative,
and $a\in V\Rightarrow |a|, a\dg\in V$. Therefore, by \cite[Theorem 5.11]{vectlat}
$V$ is a lattice.
\end{proof}

\begin{corollary} \label{co:KadTh1}
Let $p,q\in P$. Then{\rm: (i)} If $p\wedge\sb{A}q$
exists, then $pq=qp$. {\rm(ii)} If $p\wedge\sb{A}q=0$, then $pq=qp=0$.
\end{corollary}

\begin{proof}
Part (i) is an obvious consequence of Theorem \ref{th:KadTh1}. To
prove (ii), assume that $p\wedge\sb{A}q=0$. Then in Lemma \ref
{Le:SeeTheorem1Kad} with $V=A$, we have $p=p\sb{1}$ and $q=q\sb{1}$,
whence $pq=0$ by part (iii) of the lemma.
\end{proof}

In the following theorem we generalize part (ii) of Corollary \ref{co:KadTh1}
to an arbitrary pair of elements $a,b\in A$. In \cite[Corollary 8]{Kad},
Kadison obtains this result and its corollary for the special case in
which $A$ is the self-adjoint part of a W$\sp{\ast}$-algebra.

\begin{theorem} \label{Th:ab=0}
If $a,b\in A$ and $a\wedge\sb{A}b=0$, then $ab=ba=0$.
\end{theorem}

\begin{proof}
Assume the hypotheses. We have $0=a\wedge\sb{A}b\leq a,b$, i.e., $a,b
\in A\sp{+}$. We shall use the notion of the \emph{generalized infimum}
of $a$ and $b$, defined and denoted by $a\sqcap b:=\frac{1}{2}(a+b-|a-b|)$ \cite
[\S 4]{vectlat}. By \cite[Lemma 4.1 (i)]{vectlat}, $a\sqcap b\leq a,b$,
whence $a\sqcap b\leq a\wedge\sb{A}b=0$. Thus, since $a,b\in A\sp{+}$,
$ab=ba=0$ by \cite[Lemma 4.4 (ii)]{vectlat}.
\end{proof}

\begin{corollary}
If $a,b\in A$, $c:=a\wedge\sb{A}b$ exists, $cCa$, and $cCb$, then $aCb$.
\end{corollary}

\begin{proof}
Assume the hypotheses. Then $0\leq a-c, b-c$, and if $k\in A$ with $k
\leq a-c, b-c$, it follows that $k+c\leq a,b$, whence $k+c\leq c$,
so $k\leq 0$. Therefore $(a-c)\wedge\sb{A}(b-c)=0$, so by Theorem
\ref{Th:ab=0}, $(a-c)(b-c)=(b-c)(a-c)=0$. Consequently, $ab=cb+ac-c\sp{2}
=bc+ca-c\sp{2}=ba$.
\end{proof}

\section{The antilattice theorem}

In this section we prove our main theorem (Theorem \ref{th:MainTh}) giving
necessary and sufficient conditions for a synaptic algebra $A$ with a complete
projection lattice $P$ to be an antilattice. In particular, Theorem \ref
{th:MainTh} shows that a synaptic algebra with a complete projection lattice
is an antilattice iff it is a factor. To begin with, we have the following.

\begin{theorem} \label{th:ALisFac}
If $A$ is an antilattice, then $A$ is a factor.
\end{theorem}

\begin{proof}Suppose that $A$ is an antilattice. It will be sufficient
to show that the only projections in $C(A)$ are $0$ and $1$ (Theorem
\ref{th:factorcondition}). Let $p\in C(A)\cap P$. We claim that $p\wedge
\sb{A}p\sp{\perp}$ exists and equals $0$. Obviously, $0\leq p,p\sp{\perp}$.
Suppose $a\in A$ and $a\leq p,p\sp{\perp}$. As $p\in C(A)$, it follows
that $a$ commutes with both $p$ and $p\sp{\perp}$, whence $pa\leq
pp\sp{\perp}=0$ and $p\sp{\perp}a\leq p\sp{\perp}p=0$, so $a=(p+p
\sp{\perp})a=pa+p\sp{\perp}a\leq 0$. Therefore, $p\wedge\sb{A}p\sp{\perp}=0$.
Consequently, if $A$ is an antilattice and $p\in C(A)\cap P$, then
either $p\leq p\sp{\perp}$ or $p\sp{\perp}\leq p$, i.e., either $p=0$
or $p=1$.
\end{proof}

\begin{lemma} \label{lm:Fac&Sym}
If $P$ is a complete OML, $A$ is a factor, and $0<p,q\in P$ with $p\perp q$,
then there exists a symmetry $t\in A$ such that $tpt\leq q$ or $tqt\leq p$.
Moreover, $t\not=\pm1$.
\end{lemma}

\begin{proof}
Assume the hypotheses. According to \cite[Lemma 8.5]{FPsymSA}, There is a
central projection $h\in C(A)\cap P$ and a symmetry $t\in A$ such that
$ph$ and a subprojection of $q$ are exchanged by $t$ and $q(1-h)$ and a
subprojection of $p$ are exchanged by $t$. Since $A$ is a factor, the
central projection $h$ is either $0$ or $1$ (Theorem \ref{th:factorcondition}),
whence $tpt\leq q$ or $tqt\leq p$. If $t=\pm1$, then $tpt=p$ and $tqt=q$, so
$p\leq q$ or $q\leq p$, and since $p\perp q$, $p=0$ or $q=0$ contradicting $0<p,q$.
\end{proof}

\begin{lemma} \label{le:abpq}
Let $a,b\in A$ with $a\wedge\sb{A}b=0$ and suppose that $p,q\in P$,
$0<\lambda, \mu\in\reals$, $\lambda p\leq a$ and $\mu q\leq b$.
Then $p\wedge\sb{A}q=pq=qp=0$.
\end{lemma}

\begin{proof}
Assume the hypotheses and let $\kappa:=\min\{\lambda, \mu\}$. Suppose
that $g\in A$ and $g\leq p,q$. Then $\kappa g\leq\kappa p\leq a$ and
$\kappa g\leq\kappa q\leq b$, whence $\kappa g\leq 0$, and therefore
$g\leq 0$. Thus $p\wedge\sb{A}q=0$, and by Corollary \ref{co:KadTh1}
(ii), $pq=qp=0$.
\end{proof}

\begin{theorem} \label{th:Projections}
Suppose that $A$ is not an antilattice. Then there are projections $p,q\in P$
with $p\perp q$, $0<p,q$, and $p\wedge\sb{A}q=pq=qp=0$.
\end{theorem}

\begin{proof}
Since $A$ is not an antilattice, there exist $c,d\in A$ such that
$c\not\leq d$, $d\not\leq c$, and $c\wedge\sb{A}d$ exists in $A$. Put
$a:=c-c\wedge\sb{A}d$ and $b:=d-c\wedge\sb{A}d$. Obviously, $0\leq a,b$,
and since $c\not\leq d$ and $d\not\leq c$, we have $0<a,b$. Suppose $k\in A$
with $k\leq a,b$. Then $k+c\wedge\sb{A}d\leq c,d$, whence $k+c\wedge\sb{A}d
\leq c\wedge\sb{A}d$, and it follows that $k\leq 0$. Therefore, $a\wedge
\sb{A}b=0$.

By Corollary \ref{co:subprojectionoid}, there exist projections $0<p,q$
and real numbers $0<\lambda, \mu$ such that $\lambda p\leq a$ and
$\mu q\leq b$. Then by Lemma \ref{le:abpq}, $p\wedge\sb{A}q=pq=qp=0$,
and since $pq=0$, we have $p\perp q$.
\end{proof}

\begin{lemma} \label{le:Existsk}
Suppose that $p\in P$, $p\not=0,1$, and there exists a symmetry
$s\in A$ that exchanges $p$ and $p\sp{\perp}$. Then there exists
$k\in A$ with $k\leq p,p\sp{\perp}$ but $k\not\leq0$.
\end{lemma}

\begin{proof}
By the hypotheses, we have $0<p,p\sp{\perp}\in P$ and there is a symmetry
$s\in A$ such that $sps=p\sp{\perp}$. Clearly, $s\not=\pm1$.
Put
\setcounter{equation}{0}
\begin{equation}
\alpha:=-\frac{5}{4},\ \beta :=-\frac{3}{4},\ \gamma:=1,\text{\ and\ }
k=2s-\alpha\sp{2}\in A.
\end{equation}
Note that $k=2s-\frac{25}{16}$, $sks=k$, and
\begin{equation} \label{eq:2}
\alpha\sp{2}-\beta\sp{2}=1,\ (\alpha+\beta)\gamma=-2, \text{\ and\ }
 \beta\sp{2}+\gamma\sp{2}=\frac{25}{16}.
\end{equation}
We have
\[
sp=p\sp{\perp}s,\ ps=sp\sp{\perp},\ sp+ps=sp+sp\sp{\perp}=s(p+p\sp{\perp})
=s1=s
\]
\begin{equation} \label{eq:3}
\text{and\ }sp\sp{\perp}+p\sp{\perp}s=ps+sp=s.
\end{equation}
Put $d:=\alpha p+\gamma s+\beta p\sp{\perp}$. Then by (\ref{eq:2}) and
(\ref{eq:3}),
\[
0\leq d\sp{2}=\alpha\sp{2}p+\alpha\gamma(ps+sp)+\gamma\sp{2}+\beta\gamma
 (sp\sp{\perp}+p\sp{\perp}s)+\beta\sp{2}p\sp{\perp}=
\]
\[
\alpha\sp{2}p+\beta\sp{2}(1-p)+\alpha\gamma s+\gamma\sp{2}+\beta\gamma s
=(\alpha\sp{2}-\beta\sp{2})p+(\alpha+\beta)\gamma s+\beta\sp{2}+\gamma\sp{2}
\]
\[
=p-2s+\frac{25}{16}=p-k,
\]
whence $0\leq p-k$, so $k\leq p$. Therefore, $0\leq s(p-k)s=sps-k=p
\sp{\perp}-k$, so $k\leq p\sp{\perp}$. But $2\not\leq \frac{25}{16}$,
whence by Lemma \ref{le:SymIneq} $2s\not\leq\frac{25}{16}$ and therefore
$k=2s-\frac{25}{16}\not\leq 0$.
\end{proof}

\begin{theorem} \label{th:p,qAntilattice}
Suppose that whenever $0<p,q\in P$ with $p\perp q$, there exists a
symmetry $t$ in $A$ such that $tpt\leq q$ or $tqt\leq p$. Then $A$
is an antilattice.
\end{theorem}

\begin{proof}
Assume the hypothesis. Aiming for a contradiction, we assume that $A$ is not
an antilattice. By Lemma \ref{th:Projections}, there are projections
$0<p,q\in P$ with $p\perp q$ and $p\wedge\sb{A}q=pq=qp=0$. Thus, by hypothesis,
there exists a symmetry $t$ in $A$ such that $tpt\leq q$ or $tqt\leq p$. By
relabeling if necessary, we can and do assume that $tpt\leq q$. Thus,
$0<tpt\in P$ and $tpt\leq q\leq p\sp{\perp}$ so $tpt\perp p$. Therefore,
$p(tpt)=(tpt)p=0$ and $p\vee tpt=p+tpt$. If $a\in A$ and $a\leq p,tpt$, then
$a\leq p,q$, and it follows that $a\leq 0$; hence $p\wedge\sb{A}tpt=0$.

Now we are going to drop down to the synaptic algebra $A\sb{1}\subseteq
A$ defined by $A\sb{1}:=(p+tpt)A(p+tpt)$ in which $u:=p+tpt$ is the unit
element. The projection lattice of $A\sb{1}$ is the interval $P[0,u]=
\{q\in P:q\leq u\}$ in $P$ and we have $0<p,tpt\in P[0,u]$ with $p(tpt)=
(tpt)p=0$ and $p+tpt=p\vee tpt=u$, whence $tpt$ is the orthocomplement of
$p$ in $A\sb{1}$. Clearly, $p\wedge\sb{A\sb{1}}tpt=0$. Put $s:=tp+pt$. Then
$s$ is a partial symmetry in $A$ with $s\sp{2}=p+tpt$, and $s(p+tpt)=
(p+ptp)s=s$, so $s$ is a symmetry in $A\sb{1}$. Moreover, $sps=tpt$.
Applying Lemma \ref{le:Existsk} to the synaptic algebra $A\sb{1}$, we
find that there exists $k\in A\sb{1}$ with $k\leq p,tpt$ but $k\not\leq0$,
contradicting $p\wedge\sb{A\sb{1}}tpt=0$.
\end{proof}

\begin{theorem} \label{th:MainTh}
Suppose that the OML $P$ of projections in $A$ is complete. Then the following
conditions are mutually equivalent:
\begin{enumerate}
\item $A$ is an antilattice.
\item $A$ is a factor.
\item If $0<p,q\in P$ with $p\perp q$, then there exists a symmetry $t$ in $A$
such  that $tpt\leq q$ or $tqt\leq p$.
\end{enumerate}
\end{theorem}

\begin{proof}
Theorem \ref{th:ALisFac} shows that (i) $\Rightarrow$ (ii), Lemma
\ref{lm:Fac&Sym} shows that (ii) $\Rightarrow$ (iii), and Theorem
\ref{th:p,qAntilattice} shows that (iii) $\Rightarrow$ (i).
\end{proof}


\begin{thebibliography}{99}

\bibitem{Alf} Alfsen, E.M., \emph{Compact Convex Sets and Boundary
Integrals}, Springer-Verlag, New York, 1971, ISBN 0-387-05090-6.

\bibitem{Beran} Beran, L., {\em Orthomodular Lattices, An Algebraic
Approach}, Mathematics and its Applications, Vol. 18, D. Reidel Publishing
Company, Dordrecht, 1985.

\bibitem{FSyn} Foulis, David J., Synaptic algebras, \emph{Math. Slovaca}
{\bf 60}, no. 5 (2010) 631-–654.

\bibitem{SROUS} Foulis, D.J. and Pulmannov\'a, S., Spectral resolution in an
order unit space, \emph{Rep. Math. Phys.} {\bf 62} (2008) 323--344.

\bibitem{FPproj} Foulis, D.J. and Pulmannov\'a, S., Projections in a synaptic
algebra, \emph{Order} {\bf 27} (2010) 235--257.

\bibitem{FPtype} Foulis, D.J. and Pulmannov\'a, S., Type-decomposition of a synaptic
algebra, \emph{Found. Phys.} {\bf 43}, no 8 (2013) 948--968.

\bibitem{FPsymSA} Foulis, D.J. and Pulmannov\'a, S.,, Symmetries in synaptic algebras,
\emph{Math. Slovaca}, {\bf 64}, no. 3 (2014) 751--776.

\bibitem{FPcom} Foulis, D.J. and Pulmannov\'a, S., Commutativity in a synaptic algebra,
\emph{Math. Slovaca}, {\bf 66}, no. 2 (2016) 469--482.

\bibitem{FPBanach} Foulis, D.J. and Pulmannov\'a, S., Banach synaptic algebras,
submitted.	arXiv:1705.01011 [math.RA]

\bibitem{FJP2proj} Foulis, D.J., Jen\v cov\'a, A., and Pulmannov\'a, S., Two
projections in a synaptic algebra, \emph{Linear Algebra Appl.} {\bf 478} (2015)
163--287.

\bibitem{FJPpande} Foulis, D.J., Jen\v{c}ov\'a, A., and Pulmannov\'a, S., A
projection and an effect in a synaptic algebra, \emph{Linear Algebra Appl.}
{\bf 485} (2015) 417-–441.

\bibitem{vectlat} Foulis, D.J., Jen\v{c}ov\'a, A., and Pulmannov\'a, S., Vector
lattices in synaptic algebras, to appear in Math. Slovaca. arXiv:1605.06987 [math.RA].

\bibitem{FJPstat}  Foulis, D.J., Jen\v cov\'a, A., and  Pulmannov\'a,
S., States and synaptic algebras, \emph{Rep. Math. Phys.}{\bf 79}(2017) 13--32. arXiv:1605.06987[math-ph].

\bibitem{FJPMSR} Foulis, D.J., Jen\v cov\'a, A., and  Pulmannov\'a,
S., Every synaptic algebra has the monotone square root property, to
appear in \emph{Positivity}. arXiv:1605.04115 [math.OA]

\bibitem{FJPLS} Foulis, D.J., Jen\v cov\'a, A., and  Pulmannov\'a,S.,
A Loomis-Sikorski theorem and functional calculus for a generalized
Hermitian algebra, to appear in \emph{ Rep. Math. Phys.}.	arXiv:1610.06208 [math.RA]

\bibitem{GGJinf} Gheondea, Aurelian, Gudder, Stanley, and Jonas, Peter,
On the infimum of quantum effects. \emph{J. Math. Phys.} {\bf 46},
no. 6 (2005) 11 pp.

\bibitem{GPBB} Gudder, S., Pulmannov\'{a}, S., Bugajski, S., and
Beltrametti, E., Convex and linear effect algebras, \emph{Rep. Math.
Phys.} {\bf 44}, No. 3 (1999) 359--379.

\bibitem{Kad} Kadison, Richard V., Order properties of bounded self-adjoint
operators, \emph{Proc. Amer. Math. Soc.} {\bf 2} (1951) 505–-510.

\bibitem{Kalm} Kalmbach, G., {\em Orthomodular Lattices}, Academic
Press, Inc., London/New York, 1983.

\bibitem{McC} McCrimmon, K. \emph{A taste of Jordan algebras},
Universitext, Springer-Verlag, New York, 2004, ISBN: 0-387-95447-3.

\bibitem{Pid} Pulmannov\'a, S., A note on ideals in synaptic algebras,
\emph{Math. Slovaca} {\bf 62}, no. 6 (2012) 1091--1104.



\end{thebibliography}
\end{document}